\documentclass[reqno,centertags,11pt,a4paper]{amsart}
\usepackage{amssymb}
\usepackage{mathrsfs}
\usepackage{amsmath, amsfonts,vmargin, enumerate}
\usepackage{color}
\usepackage{amsthm}
\usepackage{latexsym,bm}
\usepackage{euscript}
\usepackage{color}
\usepackage{stmaryrd}
\usepackage[colorlinks=true,citecolor=blue,linkcolor=blue]{hyperref}

\begin{document}
	\newtheorem{theorem}{Theorem}
	\newtheorem{proposition}[theorem]{Proposition}
	\newtheorem{conjecture}[theorem]{Conjecture}
	\newtheorem{corollary}[theorem]{Corollary}
	\newtheorem{lemma}[theorem]{Lemma}
	\newtheorem{sublemma}[theorem]{Sublemma}
	\newtheorem{observation}[theorem]{Observation}
	\newtheorem{remark}[theorem]{Remark}
	\newtheorem{definition}[theorem]{Definition}
	\theoremstyle{definition}
	\newtheorem{notation}[theorem]{Notation}
	\newtheorem{question}[theorem]{Question}
	\newtheorem{example}[theorem]{Example}
	\newtheorem{problem}[theorem]{Problem}
	\newtheorem{exercise}[theorem]{Exercise}
	\numberwithin{theorem}{section} 
	\numberwithin{equation}{section}
	
	\title[Semi-commutative discrete Hardy-Littlewood maximal operators]{Dimension-free estimates for   semi-commutative discrete Hardy-Littlewood maximal operators}

	\author{Xudong Lai}
\address{Institute for Advanced Study in Mathematics, Harbin Institute of Technology, 150001 Harbin, China; Zhengzhou Research Institute\\
	Harbin Institute of Technology\\
	Zhengzhou
	450000\\
	China}
\email{xudonglai@hit.edu.cn}

\author{Yue Zhang}
\address{Institute for Advanced Study in Mathematics, Harbin Institute of Technology, 150001 Harbin, China}
\email{21b912030@stu.hit.edu.cn}
	
	%

	\thanks{X. Lai is supported by National Natural Science Foundation of China (No. 12271124, No. 12322107 and No. W2441002) and Heilongjiang Provincial Natural Science Foundation of China (YQ2022A005).}

	\date{}
	\maketitle


	\begin{abstract}
			For $2\leq p\leq \infty$, we establish  dimension-free   estimates for    discrete dyadic Hardy-Littlewood maximal operators over   Euclidean balls  on  semi-commutative $L_{p}$ space.   In particular, when the radius is sufficiently large, these operators admit dimension-free $L_{p}$ bounds for all $1<p<\infty$. As applications,  we derive the corresponding maximal ergodic inequalities and the bilaterally almost uniform convergence.
	\end{abstract}
	\maketitle

	\section{Introduction}
	In the 1980s,  Stein \cite{Stein}   initiated the research  on  dimension-free estimates for the continuous Hardy-Littlewood maximal operator over Euclidean balls. More precisely, for $x\in\mathbb{R}^d$ and $t>0$,   define the averaging operator as  
	 \begin{align*} 
		\mathbf{M}^{B}_{t}f(x)=\frac{1}{|B|}\int_{B}f(x-ty)dy,\quad f\in L_{Loc}^{1}(\mathbb{R}^d),
	\end{align*} 
   where  $B$  is the   unit ball. Then, there exists a constant  $C_{p}>0$  independent of  the dimension $d$ such that 
	 \begin{align*} 
	 	\Big\|\sup_{t>0}\mathbf{M}^{B}_{t}f\Big\|_{L_{p}(\mathbb{R}^d)}\leq C_{p}\|f\|_{L_{p}(\mathbb{R}^d)},\quad 1<p<\infty.
	 \end{align*}
 Shortly afterwards, maximal operators over general symmetric convex bodies  were studied (see \cite{bourgain1986high,bourgain2014hardy,carbery1986almost,muller1990geometric}). For an exhaustive exposition of this topic, we refer the reader to \cite{MR3771542}. 
	 
	    However, the discrete analogues of such dimension-free estimates were not explored until 2019, as the corresponding phenomenon is not as broad  as in the continuous setting.  Let $G$ be a bounded, closed, and symmetric convex subset of $\mathbb{R}^d$ with  non-empty interior.  For $x\in\mathbb{Z}^d$ and $t>0$, we define the discrete Hardy-Littlewood averaging operator as
		\begin{align}\label{250224.2}
		\mathcal{M}^{G}_{t}f(x)=\frac{1}{|G_{t}\cap \mathbb{Z}^d|}\sum_{y\in G_{t}\cap \mathbb{Z}^d} f(x-y),\quad f\in \ell_{1}(\mathbb{Z}^d), 
	\end{align}
	where $G_{t}=\left\{y\in\mathbb{R}^d:t^{-1}y\in G\right\}$ denotes the dilation of  $G$.  Given $1<p<\infty$,   the optimal  constant $C_{p}(d,G)>0$ satisfying
	\begin{align*}  
	\Big\|\sup_{t>0}\mathcal{M}_{t}^{G}f\Big\|_{\ell_{p}(\mathbb{Z}^d)}\leq C_{p}(d,G)\|f\|_{\ell_{p}(\mathbb{Z}^d)}, 
	\end{align*}
is referred to as the $(p,p)$-norm of the discrete Hardy-Littlewood maximal operator. The classical result stated that   $C_{p}(d,G)$  is finite and depends  on  the dimension. A natural question is whether the constant $C_{p}(d,G)$
 is independent of   the dimension $d$ for some  symmetric convex body $G$. A logical starting point is to investigate  $B^q$ balls, a fundamental example of symmetric convex sets,   defined as:
	\begin{align*}
		 &B^q=\left\{x=(x_{1},\cdots,x_{d})\in \mathbb{R}^d: |x|_{q}=\bigg(\sum_{k=1}^{d}|x_{k}|^q\bigg)^{\frac{1}{q}}\leq 1\right\},\quad1\leq q<\infty;\\
		 &B^\infty=\left\{x=(x_{1},\cdots,x_{d})\in \mathbb{R}^d:|x|_{\infty}=\max_{1\leq k\leq d}|x_{k}|\leq 1 \right\}.
	\end{align*}
 Note that $B^2$ coincides with the Euclidean unit ball $B$ mentioned earlier, and we will continue to use the notation $B$ throughout the text.  Recall that   for $1\leq q<\infty$, the   $L_{p}$ norm of the continuous Hardy-Littlewood maximal operator over $B^{q}$  is independent of the  dimension $d$ when $1<p\leq \infty$ (see \cite{muller1990geometric}). The case   $q=\infty$  has been renewed by further advances especially due to Bourgain \cite{bourgain2014hardy}. In the discrete setting, a major challenge in establishing dimension-free bounds is the lack of precise error estimates for the number of lattice points in $B^{q}$. Recent progress has been achieved only for for $B$ and $B^{\infty}$ (see \cite{MR4175747,MR3992568}). Through  sophisticated analysis,  the  $\ell_{p}(\mathbb{Z}^d)$ norm of the discrete    Hardy-Littlewood maximal operator, with the supremum   restricted to the dyadic set $ \mathbb{D}=\left\{ 2^{n}:n\in\mathbb{N}\cup  \left\{0\right\}   \right\}$  is  dimension-free  for $2\leq p\leq  \infty $.  For simplicity, we denote $Q_{t}=B^{\infty}_{t}=[-t,t]^d$ and $Q =[-1,1]^d$. Then $Q_{t}$ is a cube and  its product structure enables us to  compute the number of lattice points in $Q_{t}$,  yielding $|Q_{t}\cap \mathbb{Z}^d|=(2\lfloor t\rfloor+1)^d$. This property distinguishes cubes from  $B^{q}$ balls for $1\leq q<\infty$ and plays a crucial  role in proving  that 
    $C_{p}(d,Q)$ is   independent of  the dimension $d$ when $\frac{3}{2}<p<\infty$. 
    

	 Based on  harmonic analysis, operator algebras, noncommutative geometry, and quantum probability, noncommutative harmonic analysis has rapidly developed  \cite{MR3079331,MR4320770,MR4048299,MR3283931,mei2007operator,xia2016characterizations,MR3778570}.  
      Among its branches, the semi-commutative harmonic analysis seems to be the easiest noncommutative theory, but it also needs new ideas.  For instance, let $\mathcal{M}$ be a von Neumann algebra and $f:\mathbb{R}^d\rightarrow \mathcal{ {M}}$ an operator-valued function belonging to   the semi-commutative $L_{p}$ space. The Hardy-Littlewood maximal operator  for a sequence of operators $( \mathbf{M}^{B}_{t}f)_{t>0}$  seems no available since we can not directly compare two operators in $\mathcal{M}$. Then, it is more subtle to establish the corresponding $L_{p}$  maximal inequality.  This obstacle was not overcome until the introduction of Pisier’s vector-valued noncommutative   spaces $L_{p}(L_{\infty}(\mathbb{R}^d)\overline{\otimes} \mathcal{M};\ell_{\infty})$. Following Pisier’s idea,  Mei \cite{mei2007operator} derived  the    semi-commutative Hardy-Littlewood maximal inequality based on
      Doob’s inequality (see \cite{MR1916654}):
      \begin{align}\label{250224.1}
      	\Big\|\sup_{t>0} \mathbf{M}^{B}_{t}f\Big\|_{L_{p}(L_{\infty}(\mathbb{R}^d)\overline{\otimes}\mathcal{M})}\leq C_{p,d} \|f\|_{L_{p}(L_{\infty}(\mathbb{R}^d)\overline{\otimes}\mathcal{M})}.
      \end{align}
    Here, the  norm $\Big\|\sup_{t>0} \mathbf{M}^{B}_{t}f\Big\|_{L_{p}(L_{\infty}(\mathbb{R}^d)\overline{\otimes}\mathcal{M})}$  is understood as $L_{p}(L_{\infty}(\mathbb{R}^d)\overline{\otimes} \mathcal{M};\ell_{\infty})$ norm of  the sequence $( \mathbf{M}^{B}_{t}f)_{t>0}$ in the noncommutative case; we  refer  the reader  to Definition \ref{250210.1} in Section \ref{241223.3} for more details.  Applying Stein’s approach (see  \cite{Stein}),  Hong \cite{MR3275742} showed that the constant $C_{p,d}$ in (\ref{250224.1}) is independent of the  dimension $d$ for $1<p<\infty$. In 2018,  the semi-commutative  Hardy-Littlewood maximal operators over  symmetric  convex body were studied \cite{hong2019pointwise},  extending   dimension-free estimates of \cite{bourgain1986high,muller1990geometric} to the semi-commutative setting.

	The purpose of this paper is  to  continue the work of  Bourgain,  Mirek, Stein, and Wr\'{o}bel \cite{MR4175747} on   dimension-free estimates for   discrete dyadic Hardy-Littlewood maximal operators over   Euclidean ball   in the semi-commutative setting. 
	
    Let $\mathcal{M}$ be a von Neumann algebra equipped with a normal semifinite faithful trace $\tau$. Define the noncommutative  space  $L_{p}(\mathcal{M})$ associated with the pairs $(\mathcal{M},\tau)$. We consider the tensor von Neumann algebra $\mathcal{N}=\ell_{\infty}(\mathbb{Z}^d)\overline{\otimes}\mathcal{M}$  equipped with the trace $\sum\otimes\tau $. Then,   $L_{p}(\mathcal{N})$ is the semi-commutative space associated with the pairs  $(\mathcal{N}, \sum\otimes\tau)$. Notice that 
	     $L_{p}(\mathcal{N})$ isometrically coincides with $L_{p}(\mathbb{Z}^d;L_{p}(\mathcal{M}))$, the  $p$-integrable  functions from $\mathbb{Z}^d$ to $L_{p}(\mathcal{M})$. Given the dyadic set $ \mathbb{D}=\left\{ 2^{n}:n\in\mathbb{N}\cup  \left\{0\right\}   \right\}$, the semi-commutative discrete Hardy-Littlewood averaging operator over $\mathbb{Z}^d\cap B_{N}$ with $N\in \mathbb{D}$ has the same form as (\ref{250224.2}) but acts on functions in $ L_{p}(\mathcal{N})$.  We present our main result below, which extends the previous findings \cite{MR4175747} to the semi-commutative setting.
	   \begin{theorem}\label{8.17.1}
	   	Let $2\leq p\leq\infty$ and $f\in L_{p}(\mathcal{N})$. There exists a constant $C_{p}>0$ independent of the dimension $d$ such that
	   	\begin{align*}
	   		\Big\|\sup_{N\in \mathbb{D}}\mathcal{M}_{N}^{B}f\Big\|_{L_{p}(\mathcal{N})}\leq C_{p}\|f\|_{L_{p}(\mathcal{N})}.
	   	\end{align*}
	   \end{theorem}
  In particular, when the radius is sufficiently large, the range of $p$ can be improved.
   \begin{theorem}\label{8.18.2}
   	Let $c_{3}>0$ and define $\mathbb{D}_{c_{3},\infty}=\left\{N \in\mathbb{D}: N\geq c_{3}d\right\}$. Then  for every $f\in L_{p}(\mathcal{N})$ with $1<p<\infty$, there exists a constant $C_{p}>0$  independent of  the dimension  $d$ such that 
   	\begin{align*}
   		\Big\|\sup_{N\in D_{c_{3},\infty}}\mathcal{M}_{N}^{B}f\Big\|_{L_{p}(\mathcal{N})}\leq C_{p}\|f\|_{L_{p}(\mathcal{N})}.
   	\end{align*}
   \end{theorem}

		As  applications of Theorem \ref{8.17.1}, we study  the    maximal ergodic inequality  in the semi-commutative setting.   Let $(X,\mu)$ be a $\sigma$-finite measure space, and  consider the  tensor von Neumann algebra $\mathcal{N}_{0}=L_{\infty}(X)\overline{\otimes} \mathcal{M}$ equipped with the trace $\mu\otimes \tau$. Denoted by $L_{p}(\mathcal{N}_{0})$  the semi-commutative space associated with the pairs $(\mathcal{N}_{0}, \mu\otimes \tau)$. Recall that $(\mathcal{N}_{0}, \mu\otimes \tau,\mathbb{Z}^d,\alpha)$ is  a	$W^{*}$-dynamical system if $\alpha:\mathbb{Z}^d \rightarrow Aut(\mathcal{N}_{0})$  is a   trace-preserving group homomorphism   in the weak   $*$-topology. For any $v\in\mathbb{Z}^d$, $\alpha(v)$ extends
		to an isometric  of $L_{p}(\mathcal{N}_{0})$ for all $1\leq p<\infty$, which is still
		denoted by  $\alpha$. Given $t>0$,  the  Hardy-Littlewood ergodic averaging operator over $G_{t}\cap\mathbb{Z}^d$  is defined by 
	\begin{align*}
		\mathcal{A}^{G}_{t}f(x)=\frac{1}{|G_{t}\cap\mathbb{Z}^d|}\sum_{y\in G_{t}\cap\mathbb{Z}^d}\alpha(y)f(x),\quad f\in L_1(\mathcal{N}_{0}).
	\end{align*}
 
  	\begin{corollary}\label{1213.1}
  			\begin{itemize}
  			\item [(i)]Let $ 2\leq p<\infty$ and $f\in L_p(\mathcal{N}_{0})$, there exists  a constant  $C_{p}>0$  independent of the dimension  $d$ such that  
  			\begin{align*}
  				\Big\|\sup_{N\in \mathbb{D}}\mathcal{A}_{N}^{B}f\Big\|_{L_p(\mathcal{N}_{0})}\leq C_{p}\|f\|_{L_p(\mathcal{N}_{0})}.
  			\end{align*}
  				\item [(ii)]  Let $ 1<p<\infty$ and $f\in L_p(\mathcal{N}_{0})$, there exists  a constant  $C_{p}>0$  independent of the dimension  $d$ such that  
  				\begin{align*}
  					\Big\|\sup_{N\in D_{c_{3},\infty}}\mathcal{A}_{N}^{B}f\Big\|_{L_p(\mathcal{N}_{0})}\leq C_{p}\|f\|_{L_p(\mathcal{N}_{0})}.
  				\end{align*}
  			\end{itemize}  
   \end{corollary}


As usual,  above maximal inequalities   imply the noncommutative analogue of  almost everywhere convergence. More precisely, define a  fixed point
space of $(\alpha(v))_{v\in\mathbb{Z}^d}$, i.e.,
\begin{align*}
	F_{0}=\left\{ x\in\mathcal{N}_{0}:\alpha(v)x=x, \quad \forall \,v\,\in \mathbb{Z}^d\right\},
\end{align*}
 which  is a
von Neumann subalgebra of $\mathcal{N}_{0}$. There exists a unique conditional expectation $F: \mathcal{N}_{0} \rightarrow F_{0}$. Moreover, for $1<p<\infty$, this conditional expectation extends naturally from $L_{p}( \mathcal{N}_{0})$ to $L_{p}({F}_{0})$ and is still denoted by $F$. 
\begin{corollary}\label{250313.3}
	Let $ 1<p<\infty$ and $f\in L_p(\mathcal{N}_{0})$,    we have
	\begin{align}\label{250716.2}
		\mathcal{A}_{N}^{B}f\xrightarrow{b.a.u}Ff \quad \mbox{as}\quad  N\rightarrow \infty.
	\end{align}
\end{corollary}
 \noindent Here, b.a.u. denotes the bilaterally almost uniform convergence. For further details, the reader is referred to Definition \ref{250121.1} in Section \ref{241223.3}

	The rest of paper is organized as follows.  Section \ref{241223.3} reviews  the noncommutative  $L_{p}$ space,   vector-valued noncommutative $L_{p}$ space and noncommutative pointwise convergence. In Section \ref{25113.1},  we prove Theorem \ref{8.17.1}, which provides the dimension-free estimates for the semi-commutative discrete dyadic Hardy-Littlewood maximal operators over the ball on $L_{p}(\mathcal{N})$ for $2\leq p\leq\infty$.  In particular, when the radius is sufficiently large, these operators admit dimension-free $L_{p}$ bounds for all $1<p<\infty$, as shown in Theorem \ref{8.18.2}.  The  final  section introduces  transference principles in the ergodic setting. As applications,  we obtain the maximal ergodic inequalities  and   bilaterally almost uniform convergence    (\ref{250716.2}), stated in Corollary \ref{1213.1} and Corollary \ref{250313.3}, respectivelys

	\vspace{0.1cm}
	
	\textbf{Notation}
	We denote $C$ as a positive  constant independent of the essential variables, not necessarily the same one in each	occurrence. Letter  $C(p,d)$ is a positive constant depending on the parameters $d$ and $p$.  $X\lesssim Y$  means $X\leq C Y$ for some positive constant $C$.  $\chi_{A}$  is the indicator function of a set $A$  and $\textbf{1}=(1,\cdots,1)$.    The torus $\mathbb{T}^d$ is a priori endowed with the periodic norm: 
	\begin{align*}
		||x||= \Big(\sum_{j=1}^d ||x_j||^2\Big)^{1/2},
	\end{align*}
	where $||x_j|| = \operatorname{dist}(x_j, \mathbb{Z})$.  We identify the $d$-dimensional torus $\mathbb{T}^d$  with the unit cube  $Q = [-{1}/{2},  {1}/{2})^d$. Note that for $\xi\in\mathbb{T}^d=[-\frac{1}{2},\frac{1}{2})^d$,  $\|\xi\|$ coincides with Euclidean norm $|\xi|$. For a function $f$ on $\mathbb{Z}^d$, its   Fourier transform   $\widehat{f}$ is given by:
	\begin{align*}
	 \widehat{f}(\xi)=\sum_{x\in\mathbb{Z}^d}f(x)e^{-2\pi \mathrm{i} \langle x,\xi \rangle},\quad \xi\in\mathbb{T}^d.
	\end{align*}
    For a function $f$ on $\mathbb{T}^d$, its inverse Fourier transform   $\mathcal{F}^{-1}f$   is defined by:
    \begin{align*}
    	(\mathcal{F}^{-1}f)(x)=\int_{\mathbb{T}^{d}}f(\xi)e^{2\pi \mathrm{i} \langle x,\xi \rangle}d\xi,\quad x\in\mathbb{Z}^d.
    \end{align*}
	 
\vspace{0.1cm}

	\section{Preliminaries}\label{241223.3}
	\textbf{2.1. Noncommutative $L_{p}$ space.} Let $\mathcal{M}$ be a von Neumann algebra equipped with a normal semifinite faithful trace $\tau$. Define 
	  $\mathcal{S}_{+}(\mathcal{M})$  as the set of  positive elements $x\in\mathcal{M}$ for which $\tau(s(x)) <\infty$, where  $s(x)$ is  the smallest projection $e$ satisfying $exe=x$. Denoted by   $\mathcal{S}(\mathcal{M})$ the linear span of  $\mathcal{S}_{+}(\mathcal{M})$.  For $1\leq p<\infty$, we define
	\begin{align*}
		\|x\|_{L_{p}(\mathcal{M})}=\big(\tau(|x|^p)\big)^{\frac{1}{p}},\quad x\in \mathcal{S}(\mathcal{M}),
	\end{align*}
	where $|x|=(x^{*}x)^{\frac{1}{2}}$ is the modulus of $x$. The pair $(\mathcal{S}(\mathcal{M}),\|\cdot\|_{L_{p}(\mathcal{M})})$ forms a normed space, whose completion is the noncommutative $L_{p}$ space associated with $(\mathcal{M}, \tau )$, denoted  by $L_{p}(\mathcal{M})$. For convenience, we set $L_{\infty}(\mathcal{M})=\mathcal{M}$ equipped with the operator norm $\|\cdot\|_{\mathcal{M}}$. We denote by $L_{p}(\mathcal{M})_{+}$
	 the positive part of  $L_{p}(\mathcal{M})$. 
	
	In this paper, we are interested in   the tensor von
	Neumann algebra $\mathcal{N}=\ell_{\infty}(\mathbb{Z}^d)\overline{\otimes}\mathcal{M}$ equipped with the tensor trace $\sum\otimes\tau$. The semi-commutative space associated with the pairs $(\mathcal{N},\sum\otimes\tau)$ is denoted by  $L_{p}(\mathcal{N})$. It is well-known that $L_{p}(\mathcal{N})$   isometrically coincides with $L_{p}\big(\mathbb{Z}^d;L_{p}(\mathcal{M})\big)$,  the Bochner  $L_{p}$ space on $\mathbb{Z}^d$  with values in $L_{p}(\mathcal{M})$. Given $f\in L_{2}(\mathcal{N})$, the   operator-valued   Plancherel formula  
	\begin{align}\label{250609.1}
		\|f\|_{L_{2}(\mathcal{N})}=\|\widehat{f}\,\|_{L_{2}(  L_{\infty}(\mathbb{T}^d)\overline{\otimes}\mathcal{M})} 
	\end{align}
	is essential for us, which is a consequence of the fact that  $L_{2}(\mathcal{N})$ is a Hilbert space.  Let  $(\Sigma,\nu)$ be a measure space, and fix $1<p<\infty$.  The following operator-valued  H\"{o}lder inequality (see \cite[Lemma 2.4]{hong2022noncommutative}) will be   used: 
	\begin{align}\label{3802}
		\int_{\Sigma}f(x)g(x)d\nu(x)\leq \bigg(\int_{\Sigma}f(x)^{p}d\nu(x)\bigg)^{\frac{1}{p}}\bigg(\int_{\Sigma}g(x)^{p'}d\nu(x)\bigg)^{\frac{1}{p'}},
	\end{align}
	where $f:\Sigma\rightarrow L_{1}(\mathcal{M})+L_{\infty}(\mathcal{M})$  and $g:\Sigma\rightarrow \mathbb{C}$ are positive functions  such that  all the integrals in (\ref{3802}) make sense. Here $\leq$ is understood as the partial order in the positive cone of $\mathcal{ {M}}$.
	
	\vspace{0.2cm}
	
	\noindent\textbf{2.2. Vector-valued noncommutative $L_{p}$ spaces.} The  noncommutative  maximal norm introduced by Pisier \cite{MR1648908} and Junge \cite{MR1916654}  is the fundamental object of this paper.
	\begin{definition}\label{250210.1}
			Let $1\leq p\leq\infty$. $L_{p}( \mathcal{M};\ell_{\infty})$ is the space of all sequences $x=(x_{n})_{n\in \mathbb{Z}}$ in $L_{p}(\mathcal{M})$ which admit   factorizations of the following form: there are $a,b\in L_{2p}(\mathcal{M}) $ and a bounded sequence $y=(y_{n})_{n\in \mathbb{Z}}\subset L_{\infty}(\mathcal{M})$ such that $x_{n}=ay_{n}b$ for all  $n\in \mathbb{Z}$.
	        The norm of $x$ in $L_{p}( \mathcal{M};\ell_{\infty})$ is given by
		\begin{align*}
			\|x\|_{L_{p}( \mathcal{M};\ell_{\infty})}=\inf\left\{\|a\|_{L_{2p}(\mathcal{M})}\sup_{ n\in \mathbb{Z}}\|y_{n}\|_{\infty}\|b\|_{L_{2p}(\mathcal{M})}\right\},
		\end{align*}
		where the infimum is taken over all factorizations of $x= (x_{n})_{n\in\mathbb{Z}}=(ay_{n}b)_{n\in\mathbb{Z}}$ as above. 
	\end{definition}
   \noindent The norm   $\|x\|_{L_{p}( \mathcal{M};\ell_{\infty})}$  is conventionally denoted by $\|\sup_{n\in\mathbb{Z}} x_{n}\|_{L_{p}(\mathcal{M})}$. Here $\sup_{n\in\mathbb{Z}} x_{n}$ is just a notation and it does not make sense in the noncommutative setting.  More generally,   Definition \ref{250210.1}  can be extended to the space $L_{p}(\mathcal{M};\ell_{\infty}(\Lambda))$ with any arbitrary  index set $\Lambda$.
	
	To better understand this definition, consider  a sequence of selfadjoint operators $x=(x_{n})_{n\in\mathbb{Z}}$ in $L_{p}(\mathcal{M})$. Then,  $x\in L_{p}( \mathcal{M};\ell_{\infty})$
	 if and only if  there exists a positive operator $a\in L_{p}(\mathcal{M})$ such that $-a\leq  x_{n}\leq a $  for all $n\in\mathbb{Z}$, and moreover,
		\begin{align}\label{1223.3}
		\Big\|\sup_{n\in\mathbb{Z}} x_{n}\Big\|_{L_{p}(\mathcal{M})}=\inf\left\{\|a\|_{L_{p}(\mathcal{M})}:~a\in L_{p}(\mathcal{M})_{+},~-a \leq  x_{n}\leq a, ~\forall \,n\in\mathbb{Z} \right\}.
	\end{align}

	\begin{theorem}[\cite{MR2276775}]\label{0.4}
		Let  $1\leq p_{0}<p<p_{1}\leq\infty$ and $0<\theta<1$ be such that $\frac{1}{p}=\frac{1-\theta}{p_{0}}+\frac{\theta}{p_{1}}$. Then the following complex interpolation holds
		\begin{align*}
			L_{p}(\mathcal{M};\ell_{\infty})=\big(L_{p_{0}}(\mathcal{M};\ell_{\infty}),(L_{p_{1}}(\mathcal{M};\ell_{\infty})\big)_{\theta}\mbox{~ with equal norms}.
		\end{align*}
	\end{theorem}
	
The following  properties of $L_{p}(\mathcal{N};\ell_{\infty})$ are frequently used in this paper. 

\begin{lemma}\label{1220.7}
	Let $1\leq p\leq\infty$ and $f=(f_{n})_{n}\in L_{p}(\mathcal{{N}};\ell_{\infty})$. 
	\begin{itemize}
		\item[(i)] For any fixed $t\in\mathbb{R}^d$,   we have 
		\begin{align*}
			\|\sup_{n}e^{2\pi \mathrm{i}\langle t,\cdot\rangle}f_{n}(\cdot)\|_{L_{p}(\mathcal{N})}=	\|\sup_{n}f_{n}\|_{L_{p}(\mathcal{N})}.
		\end{align*}
		\item[(ii)] Given a bounded sequence  $(\beta_{n})_{n}\in\mathbb{C}$,  we have 
		\begin{align*}
			\|\sup_{n}\beta_{n}f_{n}\|_{L_{p}(\mathcal{N})}\leq 	\sup_{n}|\beta_{n}|\cdot\|\sup_{n}f_{n}\|_{L_{p}(\mathcal{N})}.
		\end{align*}
	\end{itemize}
\end{lemma}
\begin{proof}
	(i) Since  $f=(f_{n})_{n}\in L_{p}(\mathcal{{N}};\ell_{\infty})$,  for any $\delta>0$, there exist $a,b\in L_{2p}(\mathcal{N})$ and $(y_{n})_{n}\subset L_{\infty}(\mathcal{N})$ such that $f_{n}(\cdot)=a(\cdot)y_{n}(\cdot)b(\cdot)$ for all $n$, with 
\begin{align*}
	\|a\|_{L_{2p}(\mathcal{N})}\sup_{n}\|y_{n}\|_{\infty}\|b\|_{L_{2p}(\mathcal{N})}\leq \|\sup_{n}f_{n}(\cdot)\|_{L_{p}(\mathcal{N})}+\delta. 
\end{align*}
For the sequence  $(e^{2\pi \mathrm{i}\langle t,\cdot\rangle}f_{n}(\cdot))_{n}$, each term admits the decomposition
\begin{align*}
	e^{2\pi \mathrm{i}\langle t,\cdot\rangle}f_{n}(\cdot)=e^{2\pi \mathrm{i}\langle t,\cdot\rangle}a(\cdot)y_{n}(\cdot)b(\cdot),\quad \forall n.
\end{align*} 
A direct computation yields
	\begin{align*}
		\|\sup_{n}e^{2\pi \mathrm{i}\langle t,\cdot\rangle}f_{n}(\cdot)\|_{L_{p}(\mathcal{N})}\leq&\|e^{2\pi \mathrm{i}\langle t,\cdot\rangle}a(\cdot)\|_{L_{2p}(\mathcal{N})}\sup_{n}\|y_{n}\|_{\infty}\|b\|_{L_{2p}(\mathcal{N})}\\
		=&\|a\|_{L_{2p}(\mathcal{N})}\sup_{n}\|y_{n}\|_{\infty}\|b\|_{L_{2p}(\mathcal{N})}\\
		\leq&\|\sup_{n}f_{n}\|_{L_{p}(\mathcal{N})}+\delta.
	\end{align*}
 Since $\delta$ is arbitrary, we conclude that 
 \begin{align*}
 		\|\sup_{n}e^{2\pi \mathrm{i}\langle t,\cdot\rangle}f_{n}(\cdot)\|_{L_{p}(\mathcal{N})}\leq \|\sup_{n}f_{n}\|_{L_{p}(\mathcal{N})}.
 \end{align*}
The same method also leads to the reverse inequality.
	
	(ii) This proof is similar to (i). Hence, we omit the details here.
\end{proof}

			To introduce  the noncommutative square function,	we  first recall the   column and row   spaces. For a finite sequence $ (x_{n})_{n}$ in $L_{p}(\mathcal{M})$ with $1\leq p\leq\infty$,    define 
		\begin{align*}
			\|(x_{n})_{n}\|_{L_{p}(\mathcal{M};\ell_{2}^{c})}=\bigg\|\bigg(\sum_{n}|x_{n}|^2\bigg)^{ {1}/{2}}\bigg\|_{L_{p}(\mathcal{M})},\quad  	\|(x_{n})_{n}\|_{L_{p}(\mathcal{M};\ell_{2}^{r})}=\bigg\|\bigg(\sum_{n}|x_{n}^{*}|^2\bigg)^{ {1}/{2}}\bigg\|_{L_{p}(\mathcal{M})}.
		\end{align*}
		Notice that if $p\neq2$, above two norms are not comparable. The column  space $L_{p}(\mathcal{M};\ell_{2}^{c})$ (resp. the row space $L_{p}(\mathcal{M};\ell_{2}^{r})$) is the   completion of all finite sequences in $L_{p}(\mathcal{M})$ with  respect to $\|\cdot\|_{L_{p}(\mathcal{M};\ell_{2}^{c})}$ (resp. $\|\cdot\|_{L_{p}(\mathcal{M};\ell_{2}^{r})}$).  
		The noncommutative square function space $L_{p}(\mathcal{M};\ell_{2}^{cr})$ is defined as follows:
		if $p\geq 2$,  $L_{p}(\mathcal{M};\ell_{2}^{cr})=L_{p}(\mathcal{M};\ell_{2}^{r})\cap L_{p}(\mathcal{M};\ell_{2}^{c})$ equipped with the intersection norm:
		\begin{align*}
			\|(x_{n})_{n}\|_{L_{p}(\mathcal{M};\ell_{2}^{cr})}=\max\left\{\|(x_{n})_{n}\|_{L_{p}(\mathcal{M};\ell_{2}^{r})},\,\|(x_{n})_{n}\|_{L_{p}(\mathcal{M};\ell_{2}^{c})}   \right\};		
		\end{align*}
		if $1\leq p< 2$,  $L_{p}(\mathcal{M};\ell_{2}^{cr})=L_{p}(\mathcal{M};\ell_{2}^{r})+ L_{p}(\mathcal{M};\ell_{2}^{c})$ equipped with the sum norm:
		\begin{align*}
			\|(x_{n})_{n}\|_{L_{p}(\mathcal{M};\ell_{2}^{cr})}=\inf\left\{\|(y_{n})_{n}\|_{L_{p}(\mathcal{M};\ell_{2}^{r})}+\|(z_{n})_{n}\|_{L_{p}(\mathcal{M};\ell_{2}^{c})}   \right\},		
		\end{align*}
		where the infimum runs over all decompositions $x_{n}=y_{n}+z_{n}$ with $y_n$ and $z_n$ in
		$L_p(\mathcal{M})$.
		
		By the definition of $L_{p}( \mathcal{M};\ell_{\infty})$ and $L_{p}(\mathcal{M};\ell_{2}^{cr})$, Hong et al.  get the following result (see \cite[Proposition 2.3]{hong2019pointwise}).

	\begin{proposition}\label{8.10.1}
		Let $1\leq p\leq \infty$ and $(x_{n})\in L_{p}( \mathcal{M};\ell_{\infty})$. Then there exists an absolute
		constant $C>0$ such that  
		\begin{align*}
			\Big\|\sup_{n} x_{n}\Big\|_{L_{p}(\mathcal{M})}\leq C\|(x_n)\|_{L_{p}(\mathcal{M};\ell_{2}^{cr})}.
		\end{align*}	
	\end{proposition}

	\vspace{0.1cm}

	\noindent\textbf{2.3. Noncommutative pointwise convergence.}
	The appropriate substitute for the usual almost everywhere convergence in the noncommutative setting is the almost uniform convergence introduced by Lance \cite{MR0428060}.
	
	\begin{definition}\label{250121.1}
			Let $(x_{n})$ be a family of elements in $L_{p}(\mathcal{M})$.
			\begin{itemize}
				\item [(i)] $(x_{n})$ is said
				to converge almost uniformly (a.u. in short) to $x$, if for any $\varepsilon>0$, there exists a projection $e\in \mathcal{M}$ such that 
					\begin{align*}
					\tau(1-e)<\varepsilon \mbox{\quad and \quad}\lim_{n\rightarrow \infty}\|(x_{n}-x)e\|_{\infty}=0.
				\end{align*}
			\item [(ii)] $(x_{n})$ is said to converge bilaterally almost uniformly  (b.a.u. in short) to $x$, if for any $\varepsilon>0$ there exists a projection $e\in \mathcal{M}$ such that
			\begin{align*}
				\tau(1-e)<\varepsilon\mbox{ \quad and \quad }\lim_{n\rightarrow\infty}\|e(x_{n}-x)e\|_{\infty}=0.
			\end{align*}
			\end{itemize}
	\end{definition}
	Obviously, if  $x_{n}\xrightarrow{a.u}  x$, then $x_{n} \xrightarrow{b.a.u}  x$. In the commutative case, both  notions of convergence in Definition \ref{250121.1} are equivalent to the usual
	almost everywhere convergence by Egorov  theorem. However, they
	are different in the noncommutative setting.
	
    In order to establish the  noncommutative ergodic theorems, it is convenient to  introduce  the closed subspace
	$L_{p}(\mathcal{M};c_{0})$  of $L_{p}( \mathcal{M};\ell_{\infty})$. $L_{p}(\mathcal{M};c_{0})$ is defined as the space of all
	sequences  $(x_{n})_{n\in \mathbb{Z}}\subset L_{p}(\mathcal{M})$ such that there exist
	$a,b\in L_{2p}(\mathcal{M}) $ and a bounded sequence $y=(y_{n})_{n\in \mathbb{Z}}\subset L_{\infty}(\mathcal{M})$ satisfying
	\begin{align*}
		x_{n}=ay_{n}b\,\,\,\forall\,n\in \mathbb{Z},\mbox{\quad and\quad} \lim_{n}\|y_n\|_{\infty}=0.
	\end{align*}
The following lemma from \cite{MR2276775} is useful in the study of noncommutative ergodic theorems. 
	\begin{lemma}\label{1227.1}
		If $1\leq p<\infty$ and $(x_{n})\in L_{p}(\mathcal{M};c_{0})$, then $x_{n}\xrightarrow{b.a.u} 0$ as $n\rightarrow\infty$.
	\end{lemma}
\vspace{0.1cm}

%
%

	\section{ Proof of Theorem \ref{8.17.1}}\label{25113.1}

In this section, we prove the dimension-free estimates for the semi-commutative discrete dyadic Hardy-Littlewood maximal operators over the Euclidean ball  on $L_{p}(\mathcal{N})$  for $2\leq p\leq \infty$. If $p=\infty$, there is nothing to do, since $\mathcal{M}_{N}^{B}$ is an averaing operator. By  Theorem \ref{0.4},  it suffices to show 
	\begin{align}\label{250208.1}
		\Big\|\sup_{N\in \mathbb{D}}\mathcal{M}_{N}^{B}f\Big\|_{L_{2}(\mathcal{N})}\lesssim\|f\|_{L_{2}(\mathcal{N})},
	\end{align}
with the implicit constant is independent of  the dimension $d$. In the following, we  analyze the maximal function  corresponding to  $\mathcal{M}_{N}^{B}$ with the supremum	restricted respectively to the sets:
	\begin{itemize}
		\item[(i)]  the small-scale case:
		\begin{align*}
			\mathbb{D}_{c_{0}}=\left\{N \in\mathbb{D}:N\leq c_{0}d^{\frac{1}{2}}\right\};
		\end{align*}
		\item[(ii)]   the intermediate-scale case:
		\begin{align*}
			\mathbb{D}_{c_{1},c_{2}}=\left\{N \in\mathbb{D}:c_{1}d^{\frac{1}{2}}\leq N\leq c_{2}d\right\};
		\end{align*}
		\item[(iii)]  the large-scale case:
		\begin{align*}
			\mathbb{D}_{c_{3},\infty}=\left\{N \in\mathbb{D}:N\geq c_{3}d\right\};
		\end{align*}
	\end{itemize}
	for some universal constants $c_{0}, c_{1}, c_{2}, c_{3}>0$. Since we are working with the dyadic numbers $\mathbb{D}$, the
	  values of $c_{0}, c_{1}, c_{2}, c_{3}$  never play a role as long as they are absolute constants. Moreover, the implied constants are allowed to depend on  $c_{0}, c_{1}, c_{2}, c_{3}$.

\subsection{Dimension free estimate for the semigroup}
     In the small-scale and intermediate-scale cases, we  use the dimension-free estimate for the semigroup $P_{t}$ on $\mathbb{Z}^d$, where $t>0$ and $P_{t}$  is a convolution operator with Fourier multiplier: 
	\begin{align*} 
		  \mathfrak{p}_{t}(\xi)=e^{-t\sum_{k=1}^{d}\sin^2(\pi \xi_{k})},\quad\xi\in \mathbb{T}^d.
	\end{align*}
	Let $ e_{1},\cdots ,e_{d} $ be the standard basis in $\mathbb{Z}^d$. For every $k\in\mathbb{N}_{d}$ and $x\in \mathbb{Z}^d$, we define 
	\begin{align*}
		\Delta_{k} f(x)=f(x)-f(x+e_{k})
	\end{align*}
	 as the discrete partial derivative on $\mathbb{Z}^d$. The adjoint of $\Delta_{k}$ is  determined by  $\Delta^{*}_{k} f(x)=f(x)-f(x-e_{k})$ and  the discrete partial Laplacian is given by
	\begin{align*} 
		\mathcal{L}_{k}f(x)=\frac{1}{4}\Delta^{*}_{k}\Delta_{k}f(x)=\frac{1}{2}f(x)-\frac{1}{4}\big(f(x+e_{k})+f(x-e_{k})\big).
	\end{align*} 
	We see that  
	\begin{align*}
		\widehat{\mathcal{L}_{k}f}(\xi)=\frac{1-\cos(2\pi \xi_{k})}{2}\widehat{f}(\xi)=\sin^{2}(\pi\xi_{k})\widehat{f}(\xi),\quad\xi\in \mathbb{T}^d.
	\end{align*}
	\begin{proposition}\label{8.8.2}
		Let $1<p<\infty$ and $f\in L_{p}(\mathcal{N})$,  there exists a constant $C_{p}>0$ independent of the dimension $d$ such that  
		\begin{align*}
			\Big\|\sup_{t>0} P_{t}f\Big\|_{L_{p}(\mathcal{N})}\leq C_{p} \|f\|_{L_{p}(\mathcal{N})}.
		\end{align*}
	\end{proposition}

	    For this proof, we recall  the maximal inequality  in \cite[Corollary 5.11]{MR2276775}. 
	    
	    \begin{lemma}\label{250806.1}
	    	 Suppose that $(T_{t})_{t>0}: \mathcal{M}\rightarrow \mathcal{M}$ is a semigroup. For every $t>0$, the operator $T_{t}$ is   linear  and  satisfies:
	    	\begin{itemize}
	    		\item[(i)] $T_{t}$ is a contraction on $\mathcal{M}:$ $\|T_{t}(x)\|_{\infty}\leq\|x\|_{\infty}$ for all $x\in \mathcal{M}$;
	    		
	    		\vspace{0.1cm}
	    		
	    		\item[(ii)] $T_{t}$ is positive\,:~$ T_{t}(x)\geq0$  if $ x\geq0$;
	    		
	    		\vspace{0.1cm}
	    		
	    		\item [(iii)] $\tau \circ T_{t}\leq \tau:$ $\tau (T_{t}(x))\leq \tau(x)$ for all $x\in L_{1}(\mathcal{M})\cap \mathcal{M}_{+}$; 
	    		
	    		\vspace{0.1cm}
	    		
	    		\item[(iv)] $T_{t}$ is symmetric relative to $\tau:$ $\tau(T_{t}(y)^{*}x)=\tau(y^{*}T_{t}(x))$ for all $x,y\in L_{2}(\mathcal{M})\cap \mathcal{M}$.
	    	\end{itemize}
	    Then, for $1<p<\infty$, we have
	    	\begin{align}\label{24127.1}
	    		\Big\|\sup_{t>0} T_{t}(x)\Big\|_{L_{p}(\mathcal{M})}\leq C_{p} 	\|x\|_{L_{p}(\mathcal{M})},\quad x \in L_{p}(\mathcal{M}).
	    	\end{align}
	    \end{lemma}

	\vspace{0.2cm}
	
	\noindent $\mathbf{Proof~of~Proposition~\ref{8.8.2}:}$ We first show that, for every $k\in \mathbb{N}_{d}$, $\mathcal{L}_{k}$  generates a semigroup $(e^{-t\mathcal{L}_{k}})_{t>0}$ satisfying (i)-(iv) of Lemma \ref{250806.1}. Denoting  $\mathcal{G}_{k}f(x)=\frac{1}{2}\big(f(x+e_{k})+f(x-e_{k})\big)$,  one has $2\mathcal{L}_{k}=I-\mathcal{G}_{k}$. Thus,
	\begin{align*} 
		e^{-t\mathcal{L}_{k}}f=e^{-\frac{t}{2}}\sum_{n=0}^{\infty}\frac{(\frac{t}{2})^{n}}{n!}(\mathcal{G}_{k})^{n}f,\quad t>0.
	\end{align*}
    The above formula, together with the positivity and symmetry of $\mathcal{G}_{k}$, 
    implies that $e^{-t\mathcal{L}_{k}}$ is positive and symmetric. Moreover, we calculate  
	\begin{align*} 
		\|e^{-t\mathcal{L}_{k}}f\|_{\infty}\leq  e^{-\frac{t}{2}}\sum_{n=0}^{\infty}\frac{(\frac{t}{2})^{n}}{n!}\|(\mathcal{G}_{k})^{n}f\|_{\infty} 
		\leq e^{-\frac{t}{2}}\sum_{n=0}^{\infty}\frac{(\frac{t}{2})^{n}}{n!}\|f\|_{\infty}
		=\|f\|_{\infty}.
	\end{align*}
Similarly, for any $f\in  L_{1}( \mathcal{N})\cap\mathcal{N}_{+}$, we conclude that 
\begin{align*}
	\sum \otimes\tau (e^{-t\mathcal{L}_{k}}   f)= \|e^{-t\mathcal{L}_{k}}f\|_{L_{1}(\mathcal{N})}
	\leq  \|f\|_{L_{1}(\mathcal{N})}= \sum \otimes\tau(f).
\end{align*}
In summary,  $(e^{-t\mathcal{L}_{k}})_{t>0}$ is a semigroup satisfying (i)-(iv). Since the operators $\mathcal{L}_{1},\cdots,\mathcal{L}_{d}$ commute pairwise, we  have
	\begin{align*} 
		e^{-t(\mathcal{L}_{1}+\cdots+\mathcal{L}_{d})}=e^{-t\mathcal{L}_{1}}\circ\cdots \circ e^{-t\mathcal{L}_{d}}, \quad t\geq 0.
	\end{align*}
Thus, $\mathcal{L}=\mathcal{L}_{1}+\cdots+\mathcal{L}_{d}$ generates a semigroup satisfying (i)-(iv). 
On the other hand, observe that
	\begin{align}\label{250313.1}
		\widehat{e^{-t\mathcal{L}}f}(\xi)=e^{-t\sum_{k=1}^{d}\sin^2(\pi \xi_{k})}\widehat{f}(\xi)=\widehat{P_{t}f}(\xi).
	\end{align}
Therefore, combing (\ref{24127.1}) with (\ref{250313.1}), we obtain
	\begin{align*}
		\Big\|\sup_{t>0} P_{t}f\Big\|_{L_{p}(\mathcal{N})}=\Big\|\sup_{t>0} e^{-t\mathcal{L}}f\Big\|_{L_{p}(\mathcal{N})}\leq C_{p}\|f\|_{L_{p}(\mathcal{N})},\quad 1<p<\infty. 
	\end{align*}
	$\hfill\square$

\subsection{The small-scale and intermediate-scale cases} To prove (\ref{250208.1}), we start by estimating the
maximal function corresponding to $\mathcal{M}_{N}^{B}$ with the supremum taken over the set $\mathbb{D}_{c_{0}}$, $\mathbb{D}_{c_{1},c_{2}}$ respectively. 

\begin{lemma}\label{1203.3}
	Let $c_{0}> 0$ and  define $\mathbb{D}_{c_0}=\left\{N \in\mathbb{D}: N\leq c_{0}d^{\frac{1}{2}}\right\}$. Then   for every $f\in L_{2}(\mathcal{N})$, there exists a constant $C>0$   independent of the dimension $d$ such that
	\begin{align}\label{250313.2}
		\Big\|\sup_{N\in \mathbb{D}_{c_{0}}}\mathcal{M}_{N}^{B}f\Big\|_{L_{2}(\mathcal{N})}\leq C \|f\|_{L_{2}(\mathcal{N})}.
	\end{align}
\end{lemma}

	We write the operator $\mathcal{M}_{N}^{B}$  in    convolution form with the kernel 
	\begin{align*}
			\mathcal{K}_{N}^{B}(x)=\frac{1}{|B_{N}\cap \mathbb{Z}^{d}|}\sum_{y\in B_{N}\cap \mathbb{Z}^{d}}\delta_y(x),
		\end{align*}
	where $\delta_{y}$ is the Dirac's delta  at $y\in \mathbb{Z}^{d}$. For any $\xi\in\mathbb{T}^d$,  the multiplier corresponding to $\mathcal{M}_{N}^{B}$  is  given by
	\begin{align*} 
			\mathfrak{m}^{B}_{N}(\xi)=\widehat{\mathcal{K}_{N}^{B}}(\xi)=\frac{1}{|B_{N}\cap \mathbb{Z}^{d}|}\sum_{x\in B_{N}\cap \mathbb{Z}^{d} }e^{2\pi \mathrm{i} \langle x,\xi\rangle}.
		\end{align*}
The  inequality (\ref{250313.2}) is approximated by  maximal functions associated with   two  suitable multipliers:
\begin{align*} 
	&\lambda_{N}^{1}(\xi)=e^{-\kappa(d,N)^2\sum_{j=1}^{d}\sin^2(\pi \xi_{j})},\quad &\mbox{if } |V_{\xi}|\leq\frac{d}{2},\\
	&\lambda_{N}^{2}(\xi)=\frac{1}{|B_{N}\cap \mathbb{Z}^{d}|}\bigg(\sum_{x\in B_{N}\cap \mathbb{Z}^{d}}(-1)^{\sum_{j=1}^{d}x_{j}}\bigg)e^{-\kappa(d,N)^2\sum_{j=1}^{d}\cos^2(\pi \xi_{j})},\quad &\mbox{if }|V_{\xi}|\geq\frac{d}{2},
\end{align*}
 where 
\begin{align*}
\kappa(d,N)=Nd^{-\frac{1}{2}}, 
\end{align*}	
and 
\begin{align*}
	V_{\xi}=\left\{j\in \mathbb{N}_{d}: \cos(2\pi \xi_{j})<0\right\}=\left\{j\in \mathbb{N}_{d}: \frac{1}{4}<|\xi_{j}|\leq\frac{1}{2}\right\},\quad\xi\in\mathbb{T}^d.
\end{align*}

 It was shown in \cite{MR4175747} that the multipliers $\lambda_{N}^{1}(\xi)$ and $\lambda_{N}^{2}(\xi)$ are approximated  to   $\mathfrak{m}^{B}_{N}(\xi)$ in some sense as follows.
\begin{proposition}[\cite{MR4175747}]\label{8.9.3}
	Let $d, N \in \mathbb{N}$ such that $N \geq 2^{\frac{9}{2}}$ and $\kappa(d,N)\leq \frac{1}{5}$. Then for every $\xi\in\mathbb{T}^d$,  there exists a constant $0<c<1$ such that 
	\begin{itemize}
		\item[(i)] \mbox{if} $|V_{\xi}|\leq\frac{d}{2}$, \mbox{then }
		\begin{align}\label{1220.4}
			|\mathfrak{m}^{B}_{N}(\xi)-\lambda_{N}^{1}(\xi)|\leq 17\min\left\{e^{-\frac{c\kappa(d,N)^2}{400}\sum_{j=1}^{d}\sin^2(\pi \xi_{j})},\,\,\kappa(d,N)^2\sum_{j=1}^{d}\sin^2(\pi \xi_{j})\right\};
		\end{align}
		\item[(ii)] \mbox{if} $|V_{\xi}|\geq\frac{d}{2}$, \mbox{then } 
		\begin{align}\label{1220.5}
			|\mathfrak{m}^{B}_{N}(\xi)-\lambda_{N}^{2}(\xi)|\leq 17\min\left\{e^{-\frac{c\kappa(d,N)^2}{400}\sum_{j=1}^{d}\cos^2(\pi \xi_{j})},\,\,\kappa(d,N)^2\sum_{j=1}^{d}\cos^2(\pi \xi_{j})\right\}.
		\end{align}
	\end{itemize}
\end{proposition}

%

\vspace{0.2cm}

	\noindent $\mathbf{Proof~of~Lemma~\ref{1203.3}:}$  We first amuse that    $ N\in \mathbb{D}_{c_{0}}$ satisfying $N\geq 2^{\frac{9}{2}}$.
		Let $f\in L_2(\mathcal{N})$ and  decompose it as $f=f_{1}+f_{2}$, where  $\widehat{f_{1}}(\xi)=\widehat{f}(\xi)\chi_{\left\{\eta\in \mathbb{T}^{d}:\,|V_{\eta}|\leq \frac{d}{2}\right\}}(\xi)$. The  triangle inequality implies that
	\begin{align}\label{250114.2}
		\Big\|\sup_{N\in \mathbb{D}_{c_{0}}}&\mathcal{M}_{N}^Bf\Big\|_{L_{2}(\mathcal{N})}=\Big\|\sup_{N\in \mathbb{D}_{c_{0}}}\mathcal{F}^{-1}(\mathfrak{m}^{B}_{N}\widehat{f}\,)\Big\|_{L_{2}(\mathcal{N})}\nonumber\\
		\leq& \sum_{k=1}^{2}\Big\|\sup_{N\in \mathbb{D}_{c_{0}}}\mathcal{F}^{-1}(\lambda_{N}^{k}\widehat{f_{k}})\Big\|_{L_{2}(\mathcal{N})}+\sum_{k=1}^{2}\Big\|\sup_{N\in \mathbb{D}_{c_{0}}}\mathcal{F}^{-1}\big((\mathfrak{m}^{B}_{N}-\lambda_{N}^{k})\widehat{f_{k}}\big)\Big\|_{L_{2}(\mathcal{N})}.
	\end{align}
Notice that    $\lambda_{N}^{1}(\xi)=\mathfrak{p}_{\kappa(d,N)^2}(\xi)$, by   Proposition  \ref{8.8.2}, we have 
	\begin{align}\label{1220.6}
		\Big\|\sup_{N\in \mathbb{D}_{c_{0}}}\mathcal{F}^{-1}(\lambda_{N}^{1}\widehat{f_{1}})\Big\|_{L_{2}(\mathcal{N})}=\Big\|\sup_{N\in \mathbb{D}_{c_{0}}}P_{\kappa(d,N)^2}f_1\Big\|_{L_{2}(\mathcal{N})}\lesssim\|f_{1}\|_{L_{2}(\mathcal{N})} \leq\|f\|_{L_{2}(\mathcal{N})},
	\end{align}
where the implicit constant is independent of the dimension $d$. We claim that 
	\begin{align}\label{24128.3}
		\Big\|\sup_{N\in \mathbb{D}_{c_{0}}}\mathcal{F}^{-1}(\lambda_{N}^{2}\widehat{f_{2}})\Big\|_{L_{2}(\mathcal{N})}\lesssim\|f\|_{L_{2}(\mathcal{N})}  
	\end{align}
is deduced from \eqref{1220.6}. Indeed,   denote $F_{2}(x)=(-1)^{\sum_{j=1}^{d} x_{j}}f_{2}(x)$  and calculate 
	\begin{align*} 
		\widehat{F_{2}}(\xi)=&\sum_{x\in \mathbb{Z}^{d}}e^{-2\pi \mathrm{i}  \langle x,\xi\rangle  }(-1)^{\sum_{j=1}^{d} x_{j}}f_{2}(x)\\
		=&\sum_{x\in \mathbb{Z}^{d}}e^{-2\pi \mathrm{i}   \langle x,\xi\rangle}(e^{\pi \mathrm{i} })^{\sum_{j=1}^{d} x_{j}}f_{2}(x)\nonumber\\
		=&\sum_{x\in \mathbb{Z}^{d}}e^{-2\pi \mathrm{i} \langle  x,\xi-{\frac{\textbf{1}}{2}}\rangle  }f_{2}(x)
		=\widehat{f_{2}}(\xi-{\frac{\textbf{1}}{2}}),
	\end{align*}
where $\textbf{1}=(1,\cdots,1)\in\mathbb{Z}^d$. Hence,
	\begin{align*} 
		 \Big\|\sup_{N\in \mathbb{D}_{c_{0}}} \mathcal{F}^{-1}(\lambda_{N}^{2}\widehat{f_{2}})&\Big\|_{L_{2}(\mathcal{N})} 
		=	\Big\|\sup_{N\in \mathbb{D}_{c_{0}}} \int_{\mathbb{T}^{d}}\lambda_{N}^{2}(\xi)\widehat{f_{2}}(\xi)e^{2\pi \mathrm{i}  \langle \cdot,\,\xi\rangle}d\xi\Big\|_{L_{2}(\mathcal{N})} \nonumber\\
		\leq&\Big\|\sup_{N\in \mathbb{D}_{c_{0}}} \int_{\mathbb{T}^{d}} e^{-\kappa(d,N)^2\sum_{j=1}^{d}\cos^2(\pi \xi_{j})}\widehat{f_{2}}(\xi)e^{2\pi \mathrm{i}  \langle \cdot,\,\xi\rangle}d\xi\Big\|_{L_{2}(\mathcal{N})}\nonumber\\
		=&\Big\|\sup_{N\in \mathbb{D}_{c_{0}}} \mathcal{F}^{-1}(\lambda_{N}^{1}\widehat{F_{2}})\cdot e^{- \pi \mathrm{i}   \langle \cdot,\mathbf{1}\rangle}\Big\|_{L_{2}(\mathcal{N})}\nonumber\\
		\leq&\Big\|\sup_{N\in \mathbb{D}_{c_{0}}} \mathcal{F}^{-1}(\lambda_{N}^{1}\widehat{F_{2}}) \Big\|_{L_{2}(\mathcal{N})}\lesssim \|f\|_{L_{2}(\mathcal{N})},
		\end{align*}
where we use Lemma \ref{1220.7}, \eqref{1220.6} and the fact that  $	\|F_{2}\|_{L_{2}(\mathcal{N})}=\|f_{2}\|_{L_{2}(\mathcal{N})}\leq \|f\|_{L_{2}(\mathcal{N})}$.
	
	Now, we turn to estimating  the second term on the right-hand side of (\ref{250114.2}), beginning with the case $k=1$.  To simplify the notation, we denote 
	\begin{align*}
		\uppercase\expandafter{\romannumeral1} =\Big\|\sup_{N\in \mathbb{D}_{c_{0}}}\mathcal{F}^{-1}\big((\mathfrak{m}^{B}_{N}-\lambda_{N}^{1})\widehat{f_{1}}\big)\Big\|_{L_{2}(\mathcal{N})}. 
	\end{align*}
   Our task is to  show 
	\begin{align}\label{24127.5}
		\uppercase\expandafter{\romannumeral1}^2 \lesssim\|f\|^2_{L_{2}(\mathcal{N})},
	\end{align}
where the implicit constant is independent of the dimension $d$. By Proposition \ref{8.10.1} and the Plancherel formula \eqref{250609.1}, we get  
	\begin{align}\label{250306.2}
		\uppercase\expandafter{\romannumeral1}^2\leq& C^2\Big\|\Big(\sum_{N\in \mathbb{D}_{c_{0}}}\big|\mathcal{F}^{-1}\big((\mathfrak{m}^{B}_{N}-\lambda_{N}^{1})\widehat{f_{1}}\big)\big|^2\Big)^{\frac{1}{2}}\Big\|^2_{L_{2}(\mathcal{N})}\nonumber\\
		= &C^2\sum_{N\in \mathbb{D}_{c_{0}}}\|\mathcal{F}^{-1}\big((\mathfrak{m}^{B}_{N}-\lambda_{N}^{1})\widehat{f_{1}}\big)\|^2_{L_{2}(\mathcal{N})}\nonumber\\
		=&C^2\sum_{N\in \mathbb{D}_{c_{0}}}\|(\mathfrak{m}^{B}_{N}-\lambda_{N}^{1})\widehat{f_{1}}\|^2_{L_{2}(  L_{\infty}(\mathbb{T}^d)\overline{\otimes}\mathcal{M})}\nonumber\\
		=&C^2\tau\int_{\mathbb{T}^{d}}\sum_{N\in \mathbb{D}_{c_{0}}}|\mathfrak{m}^{B}_{N}(\xi)-\lambda_{N}^{1}(\xi)|^{2}|\widehat{f_{1}}(\xi)|^2d\xi.  
	\end{align}
	Denoting $2^{\frac{9}{2}}\leq N=2^{m}\in\mathbb{D}_{c_{0}}$ with $ m\in \mathbb{N}$   and using Proposition \ref{8.9.3},   the right-hand side of (\ref{250306.2}) is bounded by 
\begin{align*}
	& 17^2 C^2 \tau\int_{\mathbb{T}^d} 
	\sum_{\substack{m \in \mathbb{N} \\ 2^{9/2} \leq m \leq \log_2(c_0 d^{1/2})}} 
	\left(
	\min\left\{
	\frac{400d}{c\,2^{2m} \sum\limits_{j=1}^{d} \sin^2(\pi \xi_j)},\,
	\frac{2^{2m} \sum\limits_{j=1}^{d} \sin^2(\pi \xi_j)}{d}
	\right\}
	\right)^2
	|\widehat{f_{1}}(\xi)|^2\, d\xi
\end{align*}
Observe that
\[
\min\left\{
\frac{400d}{c\,2^{2m} \sum_{j=1}^{d} \sin^{2}(\pi\xi_{j})},\,
\frac{2^{2m} \sum_{j=1}^{d} \sin^{2}(\pi\xi_{j})}{d}
\right\}
\]
\[
=
\left\{
\begin{array}{ll}
	\displaystyle
	\frac{400d}{c\,2^{2m} \sum_{j=1}^{d} \sin^{2}(\pi\xi_{j})}, &
	\displaystyle
	m \geq \log_{2}\left(\frac{20d}{\sqrt{c} \sum_{j=1}^{d} \sin^{2}(\pi\xi_{j})} \right)^{1/2}, \\[10pt]
	\displaystyle
	\frac{2^{2m} \sum_{j=1}^{d} \sin^{2}(\pi\xi_{j})}{d}, &
	\displaystyle
	m < \log_{2}\left(\frac{20d}{\sqrt{c} \sum_{j=1}^{d} \sin^{2}(\pi\xi_{j})} \right)^{1/2}.
\end{array}
\right.
\]
Thus,	if 
	\begin{align*}
		\log_{2}\bigg(\frac{20d}{ \sqrt{c}\sum_{j=1}^{d}\sin^{2}(\pi\xi_{j})}\bigg)^{\frac{1}{2}}\leq \log_{2}c_{0}d^{\frac{1}{2}},
	\end{align*}
	 we   get
	\begin{align*}
	\uppercase\expandafter{\romannumeral1}^2
		\leq&17^2C^2\tau\int_{\mathbb{T}^{d}}\sum_{m\in\mathbb{N} \atop 2^{\frac{9}{2}}\leq m\leq \log_{2}\big(\frac{20d}{ \sqrt{c}\sum_{j=1}^{d}\sin^{2}(\pi\xi_{j})}\big)^{\frac{1}{2}}}\bigg(\frac{2^{2m}\sum_{j=1}^{d}\sin^{2}(\pi\xi_{j})}{d}\bigg)^2|\widehat{f_{1}}(\xi)|^2d\xi\\
		+&17^2C^2\tau\int_{\mathbb{T}^{d}}\sum_{m\in\mathbb{N} \atop \log_{2}\big(\frac{20d}{ \sqrt{c}\sum_{j=1}^{d}\sin^{2}(\pi\xi_{j})}\big)^{\frac{1}{2}}<m\leq \log_{2}c_{0}d^{\frac{1}{2}}}\bigg(\frac{400d}{c2^{2m}\sum_{j=1}^{d}\sin^{2}(\pi\xi_{j})}\bigg)^2|\widehat{f_{1}}(\xi)|^2d\xi\\
		\leq &2\cdot17^2C^2\frac{1280}{3c}\|f\|^2_{L_{2}(\mathcal{N})}.
	\end{align*}
	Otherwise,
	\begin{align*}
	\uppercase\expandafter{\romannumeral1}^2
		\leq 17^2C^2\tau\int_{\mathbb{T}^{d}}\sum_{m\in\mathbb{N} \atop 2^{\frac{9}{2}}\leq m\leq \log_{2} c_{0}d^\frac{1}{2}}\bigg(\frac{2^{2m}\sum_{j=1}^{d}\sin^{2}(\pi\xi_{j})}{d}\bigg)^2|\widehat{f_{1}}(\xi)|^2d\xi 
		\leq 17^2C^2\frac{1280}{3c}\|f\|^2_{L_{2}(\mathcal{N})}.
	\end{align*}
Combining the above two cases, (\ref{24127.5}) is proved. The same method also deduces that
\begin{align*}
	\Big\|\sup_{N\in \mathbb{D}_{c_{0}}}\mathcal{F}^{-1}\big((\mathfrak{m}^{B}_{N}-\lambda_{N}^{2})\widehat{f_{2}}\big)\Big\|_{L_{2}(\mathcal{N})}\lesssim \|f\|_{L_{2}(\mathcal{N})},
\end{align*}
with the implicit constant independent of the dimension $d$.  

Now we turn to the remaining case, namely $N\in \left\{ 2,4,8,16 \right\}$.  It is clear that 
\begin{align*}
	\Big\|\sup_{N\in \left\{ 2,4,8,16 \right\}}&\mathcal{M}_{N}^Bf\Big\|_{L_{2}(\mathcal{N})}\leq 2C\|f\|_{L_{2}(\mathcal{N})}.
\end{align*}
Indeed, by Proposition \ref{8.10.1}, we have 
\begin{align*}
	\Big\|\sup_{N\in \left\{ 2,4,8,16 \right\}}&\mathcal{M}_{N}^Bf\Big\|_{L_{2}(\mathcal{N})}\leq C\Big(\sum_{N\in \left\{ 2,4,8,16 \right\}}\| \mathcal{M}_{N}^Bf\|_{L_{2}(\mathcal{N})}^2\Big)^{\frac{1}{2}}\leq 2C\|f\|_{L_{2}(\mathcal{N})}.
\end{align*}
This proof is completed.  	$\hfill\square$

		\begin{lemma}\label{8.18.1}
		Let $c_{1},c_{2}>0$ and define  $\mathbb{D}_{c_{1},c_{2}}=\left\{N \in\mathbb{D}:c_{1}d^{\frac{1}{2}}\leq N\leq c_{2}d\right\}$. Then for every $f\in L_{2}(\mathcal{N})$, there exists a constant $C> 0$ independent of the dimension $d$ such that
		\begin{align*}
			\Big\|\sup_{N\in \mathbb{D}_{c_{1},c_{2}}}\mathcal{M}_{N}^{B}f\Big\|_{L_{2}(\mathcal{N})} \leq C \|f\|_{L_{2}(\mathcal{N})}.
		\end{align*}
	\end{lemma}
	
	 The strategy of Lemma \ref{8.18.1} is similar to the proof of  Lemma \ref{1203.3}. We omit the details here. The estimates of the multiplier $\mathfrak{m}^{B}_{N}(\xi)$ at the origin and at infinity, as shown in Proposition \ref{8.9.1} and Proposition \ref{8.9.2}, play a crucial role.

	\begin{proposition}[\cite{MR4175747}]\label{8.9.1}
		Let $d, N \in \mathbb{N}$, then for every $\xi\in\mathbb{T}^d$, we have
		\begin{align*}
			|\mathfrak{m}^{B}_{N}(\xi)-1|\leq 2\pi^2 \kappa(d,N)^2\|\xi\|^2.
		\end{align*}
	\end{proposition}
	\begin{proposition}[\cite{MR4175747}]\label{8.9.2}
		Let $d, N \in \mathbb{N}$. If $10\leq k(d, N)\leq50d^{\frac{1}{2}}$, then there exists a constant $C>0$ such that for all $\xi\in\mathbb{T}^d$, we have
		\begin{align*}
			|\mathfrak{m}^{B}_{N}(\xi)|\leq C\big( \kappa(d,N)^{-1}\|\xi\|^{-1} +\kappa(d,N)^{-\frac{1}{7}}\big).
		\end{align*}
	\end{proposition}

	\begin{remark}
		 Let $C_{1},C_{2}>0$ and  define $\mathbb{D}_{C_{1},C_{2}}=\left\{N \in\mathbb{D}:C_{1}d^{\frac{1}{q}}\leq N\leq C_{2}d\right\}$. Fix $2\leq q<\infty$, then there exists a constant $C_{q}> 0$ independent of the dimension $d$ such that for every $f\in L_{2}(\mathcal{N})$, we have 
		\begin{align*}
			\Big\|\sup_{N\in \mathbb{D}_{C_{1},C_{2}}}\mathcal{M}_{N}^{B^{q}}f\Big\|_{L_{2}(\mathcal{N})} \leq C_{q} \|f\|_{L_{2}(\mathcal{N})}.
		\end{align*}
	\end{remark}
	\begin{proof}
		Using a strategy similar to that in Lemma \ref{1203.3}, and applying \cite[Propositions 4.1 and 4.2]{MR4591824}, we obtain the desired remark.
	\end{proof}
	\subsection{ The large scale case}
	This subsection aims to proving  Theorem \ref{8.18.2}, i.e.,  the dimension-free estimates for the  maximal functions corresponding to $\mathcal{M}_{N}^{B}$ with the supremum taken over the set $\mathbb{D}_{c_{3},\infty}$. The following estimate for the number of lattice points in $B_{N}$  plays a key role.
\begin{lemma}\label{24128.1}
	There exist constants $C_{1},C_{2}>0$  such that for every $N\geq C_{1}d$, we have
	\begin{align*}
		C_{2}^{-1}|B_{N}|\leq|B_{N}\cap\mathbb{Z}^d|\leq 2e^{\frac{1}{8C_{1}^2}}|B_{N}|.
	\end{align*}
\end{lemma}

	\vspace{0.2cm}

	 \noindent $\mathbf{Proof~of~Theorem~\ref{8.18.2}:}$
		Given a function $f: \mathbb{Z}^d \rightarrow L_{p}(\mathcal{M})$, we define its extension to  $ \mathbb{R}^d$   by setting 
		\begin{align}\label{250312.2}
			F(x)=\sum_{y\in \mathbb{Z}^d}f(y)\chi_{y+Q_{\frac{1}{2}}}(x),
		\end{align}
	where $Q_{\frac{1}{2}}=[-\frac{1}{2},\frac{1}{2}]^d$.  One can  easily   verify  that 
		\begin{align*}
		F(x)=f(x),~~\forall x\in\mathbb{Z}^d,\quad \mbox{and}\quad	\|F\|_{L_{p}( L_{\infty}(\mathbb{R}^d)\overline{\otimes}\mathcal{M})}=\|f\|_{L_{p}(\mathcal{N})}.
		\end{align*}
 Without loss of generality,  we assume that $f$ is positive, which in turn implies the positivity of $F$.

	 For every $N\geq C_{1}d$, with $C_{1}$ as in Lemma \ref{24128.1}, we denote $N_{1}=(N^2+\frac{d}{4})^{\frac{1}{2}}$.  Then for $z\in Q_{\frac{1}{2}}$ and $y\in B_{N}$, it is obvious that
		\begin{align*} 
			|y+z|^2=|y|^2+|z|^2+2\langle y,z\rangle  \leq N_{1}^2 
		\end{align*}
		on the set $\left\{z\in Q_{\frac{1}{2}}: \langle y,z\rangle \leq 0\right\}$, which has measure $\frac{1}{2}$.  For all $x\in \mathbb{Z}^d$, we have
		\begin{align}\label{241021.8}
			\mathcal{M}_{N}^{B}f(x)=&\frac{1}{|B_{N}\cap \mathbb{Z}^{d}|}\sum_{y\in B_{N}\cap \mathbb{Z}^{d}}f(x+y)\chi_{B_{N}}(y)\nonumber\\
			=&\frac{2}{|B_{N}\cap \mathbb{Z}^{d}|}\sum_{y\in B_{N}\cap \mathbb{Z}^{d}}f(x+y)\int\chi_{\left\{z\in Q_{\frac{1}{2}}: \langle y,z\rangle \leq 0\right\}}(z)dz\nonumber\\
			\leq & \frac{2}{|B_{N}\cap \mathbb{Z}^{d}|}\sum_{y\in B_{N}\cap \mathbb{Z}^{d}}f(x+y)\int\chi_{\left\{z\in Q_{\frac{1}{2}}: |y+z|\leq N_{1}\right\}}(z)dz\nonumber\\
			\leq&\frac{2}{|B_{N}\cap \mathbb{Z}^{d}|}\sum_{y\in  \mathbb{Z}^{d}}f(x+y)\int_{Q_{\frac{1}{2}}}\chi_{B_{N_{1}}}(y+z)dz.
			\end{align}
	By Lemma \ref{24128.1} and a change of variables, the right-hand side of (\ref{241021.8}) is bounded by 
		\begin{align}\label{250209.7}
			 	\frac{2C_{2}}{|B_{N}|}\sum_{y\in \mathbb{Z}^{d}}f(x+y)\int_{Q_{\frac{1}{2}}}&\chi_{B_{N_{1}}}(y+z)dz
			=\frac{2C_{2}}{|B_{N}|}\sum_{y\in  \mathbb{Z}^{d}}f(y)\int_{x+B_{N_{1}}}\chi_{y+Q_{\frac{1}{2}}}(z)dz.
		\end{align}
Recalling  function $F$ given in (\ref{250312.2}) and combining (\ref{241021.8}) with (\ref{250209.7}), we obtain
\begin{align}\label{250312.1}
		\mathcal{M}_{N}^{B}f(x)\leq &\frac{2C_{2}}{|B_{N}|}\int_{x+B_{N_{1}}}F(z)dz\nonumber\\
		=&2C_{2}{\bigg(\frac{N_{1}}{N}\bigg)}^{d}\frac{1}{|B_{N_{1}}|}\int_{B_{N_{1}}}F(x+z)dz\nonumber\\
		=&2C_{2}{\bigg(\frac{N_{1}}{N}\bigg)}^{d} \mathbf{M}_{N_{1}}^{B}F(x)
		\leq 2C_{2}e^{\frac{1}{8C_{1}^2}} \mathbf{M}_{N_{1}}^{B}F(x),
\end{align}
where the last inequality follows from
		\begin{align*} 
			{\bigg(\frac{N_{1}}{N}\bigg)}^{d}=\bigg(1+\frac{d}{4N^2}\bigg)^{\frac{d}{2}}\leq \bigg(1+\frac{1}{4NC_{1}}\bigg)^{\frac{N}{2C_{1}}}\leq e^{\frac{1}{8C_{1}^2}}.
		\end{align*}

		Now, taking $N_{2}=(N_{1}^2+\frac{d}{4})^{\frac{1}{2}}$, then for $y\in Q_{\frac{1}{2}}$ and $z\in B_{N_{1}}$,  one has
		\begin{align*}
			|y+z|^2=|y|^2+|z|^2+2\langle z,y\rangle\leq N_{2}^2
		\end{align*}
		on the set $\left\{y\in Q_{\frac{1}{2}}: \langle z,y\rangle\leq 0\right\}$, which has measure $\frac{1}{2}$. Arguing in a similar way as proving (\ref{241021.8}) and using Fubini's theorem, we see that 
		\begin{align}\label{250209.1}
			 \mathbf{M}_{N_{1}}^{B}F(x)
			\leq&\frac{2}{|B_{N_{1}}|}\int_{\mathbb{R}^d}F(x+z)\int_{Q_{\frac{1}{2}}}\chi_{B_{N_{2}}}(z+y)dydz\nonumber\\
			= & {\bigg(\frac{N_{2}}{N_{1}}\bigg)}^{d}\frac{2}{|B_{N_{2}}|}\int_{Q_{\frac{1}{2}}}\int_{\mathbb{R}^d}F(x+z-y)\chi_{B_{N_{2}}}(z)dzdy \nonumber\\
			\leq &2e^{\frac{1}{8C_{1}^2}}\int_{x+Q_{\frac{1}{2}}} \mathbf{M}_{N_{2}}^{B}F(y)dy.
			\end{align}
	In view of (\ref{250312.1}) and (\ref{250209.1}), we arrive at
	\begin{align}\label{250209.5}
			\mathcal{M}_{N}^{B}f(x)\leq  4C_{2}e^{\frac{1}{4C_{1}^2}}\int_{x+Q_{\frac{1}{2}}} \mathbf{M}_{N_{2}}^{B}F(y)dy. 
	\end{align}
 By the dimension-free estimate of the semi-commutative continuous Hardy-Littlewood maximal inequality on $L_{p}( L_{\infty}(\mathbb{R}^d)\overline{\otimes}\mathcal{M})$ (see \cite[Theorem 1.2]{MR3275742}),   there exist a positive  function  $G\in L_p(L_{\infty}(\mathbb{R}^d)\overline{\otimes}\mathcal{M})$ and a constant $C_{p}>0$ independent of  the dimension $d$  such that  
		\begin{align}\label{250209.3}
			 \mathbf{M}_{t}^{B}F(y)\leq G(y),\quad\forall~ t>0,
		\end{align}
		with 
		\begin{align*} 
			\|G\|_{L_p(L_{\infty}(\mathbb{R}^d)\overline{\otimes}\mathcal{M})}\leq C_{p}\|F\|_{L_p(L_{\infty}(\mathbb{R}^d)\overline{\otimes}\mathcal{M})}.
		\end{align*}
 Therefore, combining (\ref{250209.5}) with (\ref{250209.3}),   and applying  (\ref{1223.3}), we have
		\begin{align}\label{250702.1}
			\Big\|\sup_{N>c_{3}d} 	\mathcal{M}_{N}^{B}f\Big\|^p_{{L_p(\mathcal{N})}}
			\leq \Big(4C_{2}e^{\frac{1}{4C_{1}^2}}\Big)^p&  \Big\|\int_{x+Q_{\frac{1}{2}}}G(y)dy\Big\|_{L_p(\mathcal{N})}^p.
		\end{align}
	Invoking (\ref{3802}), it follows that
\begin{align*}
	\int_{x+Q_{\frac{1}{2}}}G(y)dy\leq \Big (\int_{x+Q_{\frac{1}{2}}}G(y)^pdy\Big)^{\frac{1}{p}},
\end{align*}
which implies that the right-hand side of (\ref{250702.1}) is bounded by 
	\begin{align*}
		 \Big\|\Big(\int_{x+Q_{\frac{1}{2}}}G(y)^pdy\Big)^{\frac{1}{p}}\Big\|_{L_p(\mathcal{N})}^p
		=& \sum_{x\in \mathbb{Z}^d}\tau\int_{x+Q_{\frac{1}{2}}}G(y)^pdy\\
		=&  \tau\int_{\mathbb{R}^d}G(y)^pdy
		= \|G\|^p_{L_p(L_{\infty}(\mathbb{R}^d)\overline{\otimes}\mathcal{M})}\\
	\leq& C_{p} \|F\|^p_{L_p(L_{\infty}(\mathbb{R}^d)\overline{\otimes}\mathcal{M})}= C_{p} \|f\|^p_{L_p(\mathcal{N})}.
	\end{align*}
	We complete the proof.
$\hfill\square$
	\begin{remark}
		Let $1<p<\infty$ and $f\in L_{p}(\mathcal{N})$. Fix $q\in(2,\infty)$, there is a constant $C_q >0$ independent of the dimension $d$ such that
		\begin{align*}
			\Big\|\sup_{N>cd}\mathcal{M}_{N}^{B^q}f\Big\|_{L_{p}(\mathcal{N})}\leq C_{q}\|f\|_{L_{p}(\mathcal{N})}.
		\end{align*}
	\end{remark}
	\begin{proof}
			Using a strategy similar to that in Theorem \ref{8.18.2}, and applying \cite[Lemma 3.1 \& Lemma 3.3]{MR4591824}, this remark is obtained. We omit the details here.
	\end{proof}

	\section{Transference principle to the ergodic setting}\label{1222.2}

	In this section, we establish  the  transference principles in the ergodic setting. The first  allows us to derive $L_p$   estimates  for the maximal ergodic  Hardy-Littlewood operators associated with $(\mathcal{A}^{G}_{t})_{t\in\mathcal{I}}$   from the  corresponding bounds for $(\mathcal{M}^{G}_{t})_{t\in\mathcal{I}}$ on $L_{p}(\mathcal{N})$, where $\mathcal{I} \subseteq (0,\infty)$. Furthermore,  the $L_{p}$ boundedness of maximal ergodic  Hardy-Littlewood operators induces the  bilaterally almost uniform convergence.

	\begin{proposition}\label{1103.3}
		Given $\mathcal{I} \subseteq (0,\infty)$, suppose that for  $g\in L_{p}(\mathcal{N})$ with  $1<p<\infty$,  there exists a constant $C_{p}>0$ independent of the dimension $d$ such that 
		\begin{align}\label{1102.1}
			\Big\|\sup_{t\in\mathcal{I}}\mathcal{M}_{t}^{G}g\Big\|_{L_{p}(\mathcal{{N}})}\leq C_{p}\|g\|_{L_{p}(\mathcal{{N}})}.
		\end{align}
		   Then, for every $f\in L_{p}( \mathcal{N}_{0})$, the following inequality holds:
		\begin{align*}
			\Big\|\sup_{t\in \mathcal{I}}\mathcal{A}_{t}^{G}f\Big\|_{L_{p}(\mathcal{N}_{0})}\leq C_{p}\|f\|_{L_{p}(\mathcal{N}_{0})}, 
		\end{align*}
	where   $C_{p}$ is the same constant as in  $(\ref{1102.1})$.
	\end{proposition}
	\begin{proof}
		For any  symmetric convex body $G \subseteq \mathbb{R}^d$, there exists a constant $c_{G}\in \mathbb{N}$ such that $G_{t}\subseteq Q_{c_{G}t}$ for all $t>0$. Fix $f\in L_{p}( \mathcal{N}_{0}),~\varepsilon>0$, and $R \in \mathbb{N}$.  For every $x\in X$,  we define a new function:
		$$
		\phi_{x}(y)=
		\begin{cases}
			\alpha(y)f(x), & y\in Q_{c_{G}R(1+\frac{\varepsilon}{d})},\\
			0,& \mbox{otherwise}.
		\end{cases} 
		$$
 Then, for every  $z\in Q_{c_{G}R}\cap\mathbb{Z}^d$ and $t<\frac{R\varepsilon}{d}$, it follows that
		\begin{align*}
			\alpha(z)\mathcal{A}_{t}^{G}f(x)=&\frac{1}{|G_{t}\cap\mathbb{Z}^d|}\sum_{y\in G_{t}\cap\mathbb{Z}^d }\alpha(z)\alpha(y)f(x)\\
			=&\frac{1}{|G_{t}\cap\mathbb{Z}^d|}\sum_{y\in G_{t}\cap\mathbb{Z}^d }\alpha(z+y)f(x)\\
			=&\frac{1}{|G_{t}\cap\mathbb{Z}^d|}\sum_{y\in G_{t}\cap\mathbb{Z}^d }\phi_{x}(z+y)=\mathcal{M}_{t}^{G}\phi_{x}(z),
		\end{align*}
	where we use the fact that $z+y\in Q_{c_{G}R(1+\frac{\varepsilon}{d})} $. Hence,
		\begin{align}\label{250219.1}
		\sum_{z\in Q_{c_{G}R}\,\cap\, \mathbb{Z}^d}\bigg\|\sup_{t\in \mathcal{I}\,\cap\,(0,\frac{R\varepsilon}{d})}\alpha(z)\mathcal{A}_{t}^{G}f\bigg\|^p_{L_{p}(\mathcal{N}_{0})}=\sum_{z\in Q_{c_{G}R}\,\cap\, \mathbb{Z}^d}\bigg\|\sup_{t\in  \mathcal{I}\,\cap\,(0,\frac{R\varepsilon}{d})}\mathcal{M}_{t}^{G}\phi_{(\cdot)}(z) \bigg\|^p_{L_{p}(\mathcal{N}_{0})}.
	\end{align}
	  From now on, we assume that $f$ is positive. Since   $\alpha(y) \in Aut(\mathcal{M})$ for all $y\in\mathbb{Z}^d$, it follows that  the function $\mathcal{M}_{t}^{G}\phi_{x}(z)$ is  selfadjoint.   Therefore, by (\ref{1223.3}) and assumption (\ref{1102.1}), there exists a positive operator-valued function  $G_{x}(\cdot)\in L_{p}(\mathcal{N})$ such that 
		\begin{align}\label{241021.001}
			- G_{x}(z)\leq \mathcal{M}_{t}^{G}\phi_{x}(z)\leq G_{x}(z),\quad \forall \,t\in \mathcal{I}, 
		\end{align}
	with 
	\begin{align}\label{241021.002}
		\|G_{x}(\cdot )\|_{L_{p}(\mathcal{N})}\leq C_{p}\|\phi_{x}(\cdot )\|_{L_{p}(\mathcal{N})}.
	\end{align}
Taking into account (\ref{241021.001}),   (\ref{250219.1}) is bounded by 
		\begin{align}\label{250116.1}
		\sum_{z\in Q_{c_{G}R}\,\cap\, \mathbb{Z}^d} \|G_{(\cdot)}(z)\|^{p}_{L_{p}(\mathcal{N}_{0})} 
			=\sum_{z\in Q_{c_{G}R}\,\cap\, \mathbb{Z}^d} \tau\int_{X}|G_{x}(z)|^{p}dx
			\leq  \int_{X}\|G_{x}(\cdot )\|_{L_{p}(\mathcal{N})}^{p}dx.
			\end{align}		
	Moreover, by (\ref{241021.002}) and the isometric automorphism of $\alpha$ on $L_{p}(\mathcal{N}_{0})$, it is clear that  
		\begin{align}\label{250224.3}
			 \int_{X}\|G_{x}(\cdot )\|&_{L_{p}(\mathcal{N})}^{p}dx\leq C_{p}^{p}\int_{X}\|\phi_{x}(\cdot )\|_{L_{p}(\mathcal{N})}^{p}dx\nonumber\\
			=& C_{p}^{p} \int_{X}\tau\sum_{z\in Q_{c_{G}R(1+\frac{\varepsilon}{d})}\cap \mathbb{Z}^d} |\alpha(z)f(x)|^{p}dx\nonumber\\
			=& C_{p}^{p}  \sum_{z\in Q_{c_{G}R(1+\frac{\varepsilon}{d})}\cap \mathbb{Z}^d}\|\alpha(z)f \|^{p}_{L_{p}(\mathcal{N}_{0})} 
			= C_{p}^{p}\sum_{z\in Q_{c_{G}R(1+\frac{\varepsilon}{d})}\cap \mathbb{Z}^d}\|f \|^{p}_{L_{p}(\mathcal{N}_{0})}.
		\end{align}
Combining (\ref{250219.1}) with (\ref{250116.1}) and (\ref{250224.3}), we obtain 
		\begin{align*}
			(2c_{G}R)^d\bigg\|\sup_{t\in  \mathcal{I}\,\cap\,(0,\frac{R\varepsilon}{d})}\alpha(z)\mathcal{A}_{t}^{G}f\bigg\|^p_{L_{p}(\mathcal{N}_{0})}\leq C_{p}^{p}\bigg(2c_{G}R\bigg(1+\frac{\varepsilon}{d}\bigg)+1\bigg)^d\|f\|^p_{L_{p}(\mathcal{N}_{0})}.
		\end{align*}
		Dividing both sides by $(2c_{G}R)^d$ yields
		\begin{align*}
			\bigg\|\sup_{t\in  \mathcal{I}\,\cap\,(0,\frac{R\varepsilon}{d})}\alpha(z)\mathcal{A}_{t}^{G}f \bigg\|^p_{L_{p}(\mathcal{N}_{0})}\leq C_{p}^{p}\bigg(\bigg(1+\frac{\varepsilon}{d}\bigg)+\frac{1}{2c_{G}R}\bigg)^d\|f\|^p_{L_{p}(\mathcal{N}_{0})}.
		\end{align*}
		  Taking $R\rightarrow \infty$ and invoking  the monotone convergence theorem,  one has
		\begin{align*}
			\Big\|\sup_{t\in \mathcal{I}}\alpha(z)\mathcal{A}_{t}^{G}f \Big\|^p_{L_{p}(\mathcal{N}_{0})}\leq C_{p}^p\bigg(1+\frac{\varepsilon}{d}\bigg)^d\|f\|^p_{L_{p}(\mathcal{N}_{0})}\leq  C_{p}^pe^{\varepsilon}\|f\|^p_{L_{p}(\mathcal{N}_{0})}.
		\end{align*}
		Letting $\varepsilon\rightarrow 0^{+}$, we  conclude that
		\begin{align}\label{250116.2}
			\Big\|\sup_{t\in \mathcal{I}}\alpha(z)\mathcal{A}_{t}^{G}f \Big\|_{L_{p}(\mathcal{N}_{0})}\leq  C_{p}\|f \|_{L_{p}(\mathcal{N}_{0})}.
		\end{align}
		 Recall that it was shown in \cite[Lemma 3.3]{MR4048299}    that
		\begin{align}\label{250116.3}
			\Big\|\sup_{t\in \mathcal{I}}f_{t}\,\Big\|_{L_{p}(\mathcal{N}_{0})}=\Big\|\sup_{t\in \mathcal{I}} \big(\alpha(z)f_{t}\big)_{t}\Big\|_{L_{p}(\mathcal{N}_{0})}.
		\end{align}
     Applying (\ref{250116.3}) to (\ref{250116.2}),  this proof is completed.
	\end{proof}	
	\begin{proposition}\label{1227.2}
		Let  $1<p<\infty$ and $f\in L_{p}( \mathcal{N}_{0})$. Suppose that
		\begin{align}\label{250217.2}
			\Big\|\sup_{t>0}\mathcal{A}_{t}^{G}f\Big\|_{L_{p}( \mathcal{N}_{0})}\lesssim\|f\|_{L_{p}( \mathcal{N}_{0})},
		\end{align}
		  then  we have
		\begin{align*} 
			\mathcal{A}_{t}^{G}f\xrightarrow{b.a.u} Ff \quad \mbox{as}\quad   t\rightarrow \infty.
		\end{align*}
	\end{proposition}
	\begin{proof}
		Given $f\in L_{p}( \mathcal{N}_{0})$, by Lemma \ref{1227.1},  it suffices to prove  $(\mathcal{A}_{t}^{G}f-Ff)_{t>0}\in L_{p}( \mathcal{N}_{0};c_{0})$. 
		
		 We  first claim that the   subset
		\begin{align*}
			S=\left\{h-\alpha(v)h:h
			\in L_{1}( \mathcal{N}_{0})\cap L_{\infty}( \mathcal{N}_{0}),\quad v\in\mathbb{Z}^d\right\}
		\end{align*}
		   is dense in $\overline{(I-F)L_{p}( \mathcal{N}_{0})}$ for all $1\leq p<\infty$. It is enough to consider the case of $p=2$, since $ L_{2}( \mathcal{N}_{0})\,\cap\, L_{p}( \mathcal{N}_{0})$ is dense in $L_{p}( \mathcal{N}_{0})$ for any $1\leq p<\infty$. This goal is to show that  if $\langle \mathfrak{h},g\rangle=0$ for all $\mathfrak{h}\in S$, then $g\in F(L_{2}( \mathcal{N}_{0}))$.
		Taking $\mathfrak{h}= h-\alpha(v)h$ with $h\in L_{1}( \mathcal{N}_{0})\cap L_{\infty}( \mathcal{N}_{0})$, $v\in\mathbb{Z}^d$, and  decomposing $g$   into the linear  combination of  two  selfadjionts parts, i.e.,  $g=g_{1}+ig_{2}$, it follows that
		\begin{align*}
			0= \langle\mathfrak{h},g \rangle =&\langle \mathfrak{h},g_{1}\rangle+i \langle \mathfrak{h},g_{2} \rangle \\
			=& \langle h-\alpha(v)h,g_{1}\rangle  +i \langle h-\alpha(v)h,g_{2} \rangle \\
			=& \langle  h,g_{1}\rangle - \langle h,\alpha(v)g_{1} \rangle +i  \langle  h,g_{2}\rangle  -i \langle h ,\alpha(v)g_{2} \rangle \\
			=& \langle  h,g-\alpha(v)g \rangle.
		\end{align*}
		Thus $g=Fg\in F(L_{2}( \mathcal{N}_{0}))$, since $h\in L_{1}( \mathcal{N}_{0})\cap L_{\infty}( \mathcal{N}_{0})$ and $v\in\mathbb{Z}^d$ are arbitrary. Applying the  density to  operator $Ff-f$,  there exists a sequence  $(f_{k})_{k}\in S$ such that
		\begin{align}\label{250704.1}
			\lim_{k\rightarrow\infty}\|f-Ff-f_{k}\|_{L_{p}( \mathcal{N}_{0})}=0.
		\end{align}
		By assumption (\ref{250217.2}), one has
		\begin{align}\label{250704.2}
		\Big\|\sup_{t>0}\big(\mathcal{A}_{t}^{G}f-Ff-\mathcal{A}_{t}^{G}f_{k}\big) \Big\|_{L_{p}( \mathcal{N}_{0})}=&	\Big\|\sup_{t>0}\mathcal{A}_{t}^{G}(f-Ff -f_{k})\Big\|_{L_{p}( \mathcal{N}_{0})}\nonumber\\
			\lesssim& \|f-Ff-f_{k}\|_{L_{p}( \mathcal{N}_{0})}.
		\end{align}
		Combing (\ref{250704.1}) and (\ref{250704.2}), we get 
		\begin{align*}
			\lim_{k\rightarrow\infty}(\mathcal{A}_{t}^{G}f_{k})_{t}=(\mathcal{A}_{t}^{G}f-Ff)_{t}\quad  \mbox{in} \quad L_{p}(  \mathcal{N}_{0};\ell_{\infty}).
		\end{align*}
	    Since $L_{p}( \mathcal{N}_{0};c_{0})$ is closed in $L_{p}(  \mathcal{N}_{0};\ell_{\infty})$,  it reduced to show 
	    $(\mathcal{A}_{t}^{G}f_{k})_{t}\in L_{p}( \mathcal{N}_{0};c_{0})$ for every $ f_{k}\in S$. 
	    
	    Fix $ f_{k}=h-\alpha(v)h\in S$ with $h\in L_{1}( \mathcal{N}_{0})\cap L_{\infty}( \mathcal{N}_{0})$ and $v\in\mathbb{Z}^d$.  Choosing $1<q<p$, $0<s_{1}<s_{2}$, and using Theorem \ref{0.4} with assumption (\ref{250217.2}), one obtains
		\begin{align}\label{250710.1}
			\Big\|\sup_{s_{1}<t<s_{2}}\mathcal{A}_{t}^{G} f_{k}\Big\|_{L_{p}( \mathcal{N}_{0})}\leq&  \Big\|\sup_{s_{1}<t<s_{2}}\mathcal{A}_{t}^{G} f_{k}\Big\|_{L_{q}( \mathcal{N}_{0})}^{\frac{q}{p}}\bigg(\sup_{s_{1}<t<s_{2}}\|\mathcal{A}_{t}^{G} f_{k}\|_{\infty}\bigg)^{1-\frac{q}{p}}\nonumber\\ 
			\lesssim&   \| f_{k}\|_{L_{q}( \mathcal{N}_{0})}^{\frac{q}{p}}\bigg(\sup_{s_{1}<t<s_{2}}\|\mathcal{A}_{t}^{G} f_{k}\|_{\infty}\bigg)^{1-\frac{q}{p}}.
		\end{align}
		Notice that
		\begin{align*}
			\mathcal{A}_{t}^{G} f_{k}=&\mathcal{A}_{t}^{G}h-\mathcal{A}_{t}^{G}\alpha(v)h\\
			=&\frac{1}{|G_{t}\cap\mathbb{Z}^d|}\sum_{y\in G_{t}\cap\mathbb{Z}^d }\alpha(y)h-\frac{1}{|G_{t}\cap\mathbb{Z}^d|}\sum_{y\in G_{t}\cap\mathbb{Z}^d }\alpha(v+y)h\\
			=&\frac{1}{|G_{t}\cap\mathbb{Z}^d|}\sum_{y\in G_{t}\cap\mathbb{Z}^d }\alpha(y)h-\frac{1}{|G_{t}\cap\mathbb{Z}^d|}\sum_{y\in G_{t}\cap\mathbb{Z}^d +v}\alpha(y)h\\
			=&\frac{\sum_{y\in\mathbb{Z}^d}\big(\chi_{G_{t}\cap\mathbb{Z}^d}(y)-\chi_{G_{t}\cap\mathbb{Z}^d+v}(y)\big)\alpha(y)h}{|G_{t}\cap\mathbb{Z}^d|}.
		\end{align*}
		Then, the triangle inequality yields that 
		\begin{align}\label{250710.2}
			\|\mathcal{A}_{t}^{G} f_{k}\|_{\infty}\leq& \frac{\sum_{y\in\mathbb{Z}^d}|\chi_{G_{t}\cap\mathbb{Z}^d}(y)-\chi_{G_{t}\cap\mathbb{Z}^d+v}(y)|\cdot\|\alpha(y)h\|_{\infty}}{|G_{t}\cap\mathbb{Z}^d|}\nonumber\\
			\leq &\frac{|\chi_{G_{t}\cap\mathbb{Z}^d}\triangle\chi_{G_{t}\cap\mathbb{Z}^d+v}|\cdot\|h\|_{\infty}}{|G_{t}\cap\mathbb{Z}^d|}\rightarrow 0 \quad \mbox{as}\quad t\rightarrow \infty.
		\end{align}
	Combining (\ref{250710.1}) with (\ref{250710.2}), we get 
	\begin{align*}
			\Big\|\sup_{s_{1}<t<s_{2}}\mathcal{A}_{t}^{G} f_{k}\Big\|_{L_{p}( \mathcal{N}_{0})}\rightarrow 0 \quad \mbox{as}\quad s_{1}\rightarrow \infty,
	\end{align*}
		which implies that $(\mathcal{A}_{t}^{G}f_{k})_{t>0}$ is approximated by $(\mathcal{A}_{t}^{G}f_{k})_{0<t<s_{1}}$ in   $L_{p}(  \mathcal{N}_{0};\ell_{\infty})$, hence belonging to  $L_{p}( \mathcal{N}_{0};c_{0})$. 
	\end{proof}

			\end{document}